\documentclass[12pt, reqno]{amsart}
\usepackage{latexsym,amsfonts,amsmath,amssymb,color,amsthm}

\usepackage{enumerate}

\usepackage[colorlinks, bookmarks=true]{hyperref}
\usepackage{color,graphicx,shortvrb}
\usepackage[active]{srcltx} 

\newtheorem{theorem}{Theorem}[section]
\newtheorem{remark}[theorem]{Remark}

\newtheorem{lemma}[theorem]{Lemma}

\newtheorem{proposition}[theorem]{Proposition}
\newtheorem{corollary}[theorem]{Corollary}

\newtheorem*{theorem*}{Theorem (Denjoy-Wolff)}
\newtheorem*{theoremA*}{Theorem A}
\newtheorem*{theoremB1*}{Theorem B1}
\newtheorem*{theoremB2*}{Theorem B2}
\newtheorem*{theoremC*}{Theorem C}

\catcode`\@=11

\renewcommand{\section}%
   {\setcounter{equation}{0}\@startsection {section}{1}{\z@}{-3.5ex plus -1ex
  minus -.2ex}{2.3ex plus .2ex}{\Large\bf}}

\numberwithin{equation}{section}

\newcommand{\C}{\mathbb{C}} 

 \DeclareMathOperator{\dist}{dist}
\DeclareMathOperator{\diam}{diam}

\usepackage[latin 1]{inputenc}

\title[Inner and weak-2-local derivations]{Inner derivations and weak-2-local derivations on the C$^*$-algebra $C_0(L,A)$}

\date{\today}

\author[E. Jord\'a]{Enrique Jord\'a}

\address{Escuela Polit\'ecnica Superior de Alcoy, IUMPA, Universitat Polit\'ecnica de Valencia, Plaza Ferr\'andiz y Carbonell 1, 03801 Alcoy, Spain}
\email{ejorda@mat.upv.es}

\author[A.M. Peralta]{Antonio M. Peralta}

\address{Departamento de An{\'a}lisis Matem{\'a}tico, Universidad de Granada,
Facultad de Ciencias 18071, Granada, Spain}
\email{aperalta@ugr.es}

\thanks{First author partially supported by the Spanish Ministry of Economy and Competitiveness Project MTM2013-43540-P. Second author partially supported by the Spanish Ministry of Economy and Competitiveness and European Regional Development Fund project no. MTM2014-58984-P and Junta de Andaluc\'{\i}a grant FQM375.}



\begin{document}
\begin{abstract} Let $L$ be a locally compact Hausdorff space. Suppose $A$ is a C$^*$-algebra with the property that every weak-2-local derivation on $A$ is a {\rm(}linear{\rm)} derivation. We prove that every weak-2-local derivation on $C_0(L,A)$ is a {\rm(}linear{\rm)} derivation. Among the consequences we establish that if $B$ is an atomic von Neumann algebra or on a compact C$^*$-algebra, then every weak-2-local derivation on $C_0(L,B)$ is a linear derivation. We further show that, for a general von Neumann algebra $M$, every 2-local derivation on $C_0(L,M)$ is a linear derivation. We also prove several results representing derivations on $C_0(L,B(H))$ and on $C_0(L,K(H))$ as inner derivations determined by multipliers.
\end{abstract}

\maketitle

\section{Introduction}

Inspired by the Kowalski-S{\l}odkowski theorem (see \cite{KoSlod}), P. \v{S}emrl introduced in \cite{Semrl97} the notions of 2-local derivations and automorphisms. This notion and subsequent generalizations have been intensively explored in recent papers (see, for example, \cite{NiPe2014,NiPe2015,CaPe2015,CaPe2016} and \cite{JorPe}). More recent contributions deal with the following general notion: Let $\mathcal{S}$ be a subset of the space $L(X,Y)$, of all linear maps between Banach spaces $X$ and $Y$, a (non-necessarily linear nor continuous) mapping $\Delta : X\to Y$ is said to be a \emph{weak-2-local $\mathcal{S}$ map} (respectively, a \emph{2-local $\mathcal{S}$-map}) if for every $x,y\in X$ and $\phi\in Y^{*}$ (respectively, for every $x,y\in X$), there exists $T_{x,y,\phi}\in \mathcal{S}$, depending on $x$, $y$ and $\phi$ (respectively, $T_{x,y}\in \mathcal{S}$, depending on $x$ and $y$), satisfying $$\phi \Delta(x) = \phi T_{x,y,\phi}(x), \hbox{ and  }\phi \Delta(y) = \phi T_{x,y,\phi}(y)$$ (respectively, $ \Delta(x) =  T_{x,y}(x),$ and $  \Delta(y) = T_{x,y}(y)$).\smallskip

The real advantage of this definition is that we can unify several notions under the same perspective. That is, if $\mathcal{S}$ is the set Der$(A)$, of all derivations on a Banach algebra $A$ (respectively, the set $\hbox{$^*$-Der}(A)$ of all $^*$-derivations on a C$^*$-algebra $A$, or, more generally, the set of all symmetric maps from $A$ into another C$^*$-algebra $B$), (weak-)2-local $\mathcal{S}$-maps are called \emph{(weak-)2-local derivations} (respectively, \emph{(weak-)2-local $^*$-derivations} or \emph{(weak-)2-local symmetric maps}). We recall that a mapping $\Delta$ from a C$^*$-algebra $A$ into a C$^*$-algebra $B$ is called \emph{symmetric} if $\Delta (a^*) =\Delta (a)^*$, for every $a\in A$.\smallskip

The linearity of a 2-local $\mathcal{S}$-map is not always guaranteed. For example, as noted in \cite{JorPe}, for $\mathcal{S}=K(X,Y)$ the space of compact linear mappings from $X$ to $Y$, every 1-homogeneous map $\Delta: X\to Y$, i.e. $\Delta(\alpha x)=\alpha \Delta(x)$ for each $\alpha\in\C$, is a 2-local $\mathcal{S}$-map. However, many interesting subsets enjoy good stability properties for (weak-)2-local perturbations. For example,  let $H$ be an infinite-dimensional separable Hilbert space. P. \v{S}emrl proved in \cite{Semrl97} that every 2-local automorphism (respectively, every 2-local derivation) on the von Neumann algebra $B(H)$, of all bounded linear operators on $H$, is an automorphism (respectively, a derivation). Sh. Ayupov and K. Kudaybergenov extended \v{S}emrl's theorem by showing that every 2-local derivation on a von Neumann algebra $M$ is a derivation (see \cite{AyuKuday2014}). The problem whether the same conclusion remains true for general C$^*$-algebras is being intensively studied. Exploring new types of C$^*$-algebras, Sh. Ayupov and F.N. Arzikulov prove that for every compact Hausdorff space $\Omega$ and for every Hilbert space $H$, each 2-local derivation on $C(\Omega,B(H))$ is a derivation (see \cite{AyuArz}).\smallskip

Recent studies show that, even under weaker hypothesis, similar conclusions remain true for other clases of weak-2-local $\mathcal{S}$-maps. M. Niazi and the second author of this note prove in \cite{NiPe2014,NiPe2015} that every weak-2-local derivation on a finite dimensional C$^*$-algebra is a linear derivation, and every weak-2-local $^*$-derivation on $B(H)$ is a linear $^*$-derivation. A generalization of this theorem is established, in collaboration with J.C. Cabello, by showing that every weak-2-local symmetric map between general C$^*$-algebras is a linear map \cite[Theorem 2.5]{CaPe2015}. Consequently, every weak-2-local $^*$-derivation on a general C$^*$-algebra is a (linear) $^*$-derivation, and  every 2-local $^*$-homomorphism between C$^*$-algebras is a (linear) $^*$-homomorphism (see \cite[Corollary 2.6]{CaPe2015}). In a more recent contribution, due to the same authors, it is established that every weak-2-local derivation on $B(H)$ or on $K(H)$ is a linear derivation, where $H$ is an arbitrary complex Hilbert space and $K(H)$ stands for the C$^*$-algebra of all compact operators on $H$. Actually, every weak-2-local derivation on an atomic von Neumann algebra or on a compact C$^*$-algebra is a linear derivation \cite{CaPe2016}. We recently enlarged the class of C$^*$-algebras $A$ satisfying that each weak-2-local derivation on $A$ is a linear derivation by showing that every C$^*$-algebra of the form $C(\Omega,B(H))$ or $C(\Omega, K(H))$ lies in this class.\smallskip

The purpose of the first part of this work is to continue with the study of weak-2-local derivations on general C$^*$-algebras. Unlike von Neumann algebras, a general C$^*$-algebra need not be unital. A prototype is a C$^*$-algebra of the form $C_0(L,A)$, where $L$ is a locally compact Hausdorff space and $A$ is any C$^*$-algebra. The main results in Section \ref{sec: weak-2-local der on CoL} prove that if every (weak-)2-local derivation on $A$ is a linear derivation, then every (weak-)2-local derivation on $C_0(L,A)$ is a linear derivation (see Theorems \ref{t pattern} and \ref{t pattern 2local}). Among the consequences we shall show that if $B$ is an atomic von Neumann algebra or a compact C$^*$-algebra, then every weak-2-local derivation on $C_0(L,B)$ is a linear derivation (see Theorem \ref{t weak2local derivations C0LB(H)}). Furthermore, for a general von Neumann algebra $M$, every 2-local derivation on $C_0(L,M)$ is a linear derivation (compare Corollary \ref{c 2local derivations C0LB(H)}).\smallskip

Previous studies on weak-2-local derivations rely on the precise representation of derivations on certain C$^*$-algebras as inner derivations. We recall that a derivation on a Banach algebra $A$ is a linear mapping $D: A\to A$ satisfying $D(ab ) = D(a) b + a D(b)$ for every $a,b\in A$. For each $z\in A$, the mapping $$\hbox{ad}_{z}: A\to A,\ \ x\mapsto \hbox{ad}_{z} (x) =[z,x] =  z x - x z,$$ is a derivation on $A$. Derivations of the form $\hbox{ad}_{z}$ are called \emph{inner derivations}. A celebrated result due to S. Sakai proves that every derivation on a von Neumann algebra is inner (see \cite[Theorem 4.1.6]{S}). The existence of C$^*$-algebras admitting derivations which are not inner is a well known fact (see \cite[Example 1.4.8]{S}). The question whether every derivation on a concrete C$^*$-algebra is inner or not gains importance after these results. For a von Neumann algebra $M$, C.A. Akemann and B.E. Johnson prove that every derivation of the C$^*$-tensor product $C(\Omega) \otimes M\cong C(\Omega, M)$ is inner. In certain cases, there exist derivations which are not inner however they are very similar to inner derivations. For example, by Proposition 2.10 in \cite{JorPe}, for each compact Hausdorff space $\Omega$ with a topology $\tau$, each complex Hilbert space $H$, and each derivation $D:C(\Omega,K(H))\to C(\Omega,K(H))$, there exists a $\tau$-weak$^*$-continuous, bounded mapping $Z_0:  \Omega\to B(H)$ satisfying $D(X)= [Z_0,X]$, for every $X\in C(\Omega,K(H))$. Remark 2.11 in \cite{JorPe} shows that the mapping $Z_0$ need not be, in general, $\tau$-norm continuous.\smallskip

In the second part of this paper we study when a derivation on a C$^*$-algebra of the form $C_0(L,A)$ can be represented by an inner derivation on its multiplier algebra. In a first result we obtain a ``local'' representation theorem for derivations on $C_0(L,M)$, where $M$ is an arbitrary von Neumann algebra and $L$ is a locally compact Hausdorff space. Concretely, let $D: C_0(L,M) \to C_0(L,M)$ be a derivation. Given $\varepsilon>0$, and a compact subset $K\subset L$, then there exists a continuous and bounded function $Z_K : K \to M$ such that $\| Z_K \| \leq ({1+2 \varepsilon}) \|D\|$ and $D(X) (t)= [Z_K , X](t)$, for every $X\in C_0(L,M)$ and every $t\in K$ (see Theorem \ref{t representation of derivations on C0LM}). Theorem \ref{t AkJohn paracompact strengtened with control on the norm} proves that under the additional hypothesis of $L$ being paracompact, then for each $\varepsilon>0$ there exists a continuous and bounded function $Z_0 : L \to M$ such that $\| Z_0 \| \leq ({1+2 \varepsilon}) \|D\|$ and $D(X) = [Z_0 , X]$, for every $X\in C_0(L,M)$. An appropriate version for derivations on $C_0(L,K(H))$ is presented in Theorem \ref{t representation of derivations on C0LKH} and Corollary \ref{c tau-strong* continuity of Z0 in K(H)}.\smallskip

Finally, after proving that for every von Neumann algebra $M$, every derivation on $C_0(L,M)$ is point-weak$^*$ continuous (see Proposition \ref{p automatic point-weak* continuity}), we establish a global representation theorem for derivations on $C_0(L,B(H))$. We show that for each derivation $D$ on $C_0(L,B(H))$, then  there exists a bounded function $Z_0 : L \to B(H)$ which is $\tau$-to-norm continuous, $\| Z_0 \| \leq 2 \|D\|$ and $D(X) = [Z_0 , X] $, for every $X\in C_0(L,B(H))$ (compare Theorems \ref{t representation of derivations on C0LBH} and \ref{thm final tau-strong* continuity of Z0}). These results are directly connected with the studies on derivations of $C_0(L,B(H))$ conducted by E.C. Lance \cite{Lanc69}, C.A. Akemann, G.E. Elliott, G.K. Pedersen and J. Tomiyama \cite{AkEllPedTomi76}.

\section{Weak-2-local derivations on C$^*$-algebras of continuous functions}\label{sec: weak-2-local der on CoL}

In this section we begin by exploring the properties of weak-2-local maps. We shall illustrate this note with an example which shows that, in general, 2-local $\mathcal{S}$-maps and weak-2-local $\mathcal{S}$-maps are strictly different classes of maps. We begin with weak-local and local maps. Let $\mathcal{S}$ be a subset of $L(X,Y)$, where $X$ and $Y$ are Banach spaces. Following \cite{EssaPeRa16, EssaPeRa16b}, we say that a linear mapping $T: X\to Y$ is a weak-local $\mathcal{S}$ map (respectively, a local $\mathcal{S}$ map) if for each $x\in X$ and $\phi\in Y^*$ (respectively, for each $x\in X$) there exists $S_{x,\phi}\in \mathcal{S}$, depending on $x$ and $\phi$ (respectively, $S_{x}\in \mathcal{S}$, depending on $x$), such that $\phi T(x) = \phi S_{x,\phi} (x)$ (respectively, $T(x) = S_{x} (x)$).

\begin{proposition}\label{p weak-local and local do not coincide} Let $X$ and $Y$ be Banach spaces with $Y$ infinite dimensional. Suppose $F$ is a proper norm-dense subspace of $Y$. Let $\mathcal{S}$ be the set of all finite rank operators $S$ in $L(X,Y)$ such that $S(X) \subset F$. Then the local $\mathcal{S}$ maps are the linear maps from $X$ to $Y$ whose image is contained in $F$, while the set of weak-local $\mathcal{S}$ maps is the whole $L(X,Y)$.
\end{proposition}

\begin{proof} To prove the first statement, let $T: X\to Y$ be a linear local $\mathcal{S}$ map. For each $x\in X$, there exists $S_{x}\in \mathcal{S}$ such that $T(x)=S_{x} (x) \in F$, which proves that $T(X)\subset F$. Suppose now that $T: X\to Y$ is a linear map such that $T(X)\subset F$. For each $x\in X$, $T(x)\in F$. We take, via Hahn-Banach theorem, a functional $\varphi_{x}\in X^*$ satisfying $\varphi_x (x) =1$. The linear map $S_x = T(x)\otimes \varphi_x$ lies in $\mathcal{S}$ and $T(x) = S_{x} (x)$. We have therefore shown that $T$ is a local $\mathcal{S}$ map.\smallskip

For the second statement, let $T: X\to Y$ be a linear map, and let us fix $x\in X$ and $\phi \in Y^*$. If $\phi=0$, then $\phi T(x) = \phi S(x)=0$, for every $S\in \mathcal{S}$. We may therefore assume that $\phi\neq 0$. The density of $F$ in $Y$, implies that $\phi (F)\neq 0$ and hence $\phi (F) = \mathbb{C}$. Let us pick $u\in F$ such that $\phi (T(x)) = \phi(u)$. As before, take a functional $\varphi_{x}\in X^*$ satisfying $\varphi_x (x) =1$. The operator $S_x = u\otimes \varphi_x$ belongs to $\mathcal{S}$ and $\phi T(x) = \phi (u) = \phi S_{x} (x)$, which shows that $T$ is a weak-local $\mathcal{S}$ map.
\end{proof}

\begin{remark}\label{r weak-local is local}{\rm Let $\mathcal{S} = \hbox{Der} (A)$ be the set of all derivations on a C$^*$-algebra $A$. Theorem 3.4 in \cite{EssaPeRa16} (see also \cite{EssaPeRa16b}) proves that every weak-local $\hbox{Der} (A)$ map on $A$ lies in $\hbox{Der} (A)$. This phenomenon also holds for other sets $\mathcal{S}$, for example when $\mathcal{S}$ is the set of all triple derivations on a JB$^*$-triple (see \cite{BurCaPe16}). To provide an example in the setting of general Banach spaces, let $X$ and $Y$ be Banach spaces with $Y$ infinite dimensional, and let $F$ be a closed proper subspace of $Y$. We set $$\mathcal{S}:= \{ S\in L(X,Y) : S\hbox{ has finite rank and } S(X)\subset F\}.$$ We claim that, under these hypothesis, every weak-local $\mathcal{S}$ map is a local $\mathcal{S}$ map. Arguing by contradiction, we suppose that $T:X\to Y$ is a weak-local $\mathcal{S}$ map which is not a local $\mathcal{S}$ map. Then there exists $x\in X$ such that $T(x) \neq S(x)$, for every $S\in \mathcal{S}$. We observe that $F= \{ S(x) : S\in \mathcal{S}\}$, and hence $T(x)\notin F$. Since $F$ is a closed subspace, we can find, via Hahn-Banach theorem, a functional $\phi\in Y^*$ satisfying $\phi T(x) =1$ and $\phi S(x) =0$, for every $S\in \mathcal{S}$, which contradicts that $T$ is a weak-local $\mathcal{S}$ map. In the latter case, the (weak-)local $\mathcal{S}$ maps are precisely the linear maps $T\in L(X,Y)$ such that $T(X)\subset F$.
}\end{remark}

We can also prove the existence of weak-2-local $\mathcal{S}$ maps which are not 2-local $\mathcal{S}$ maps.

\begin{proposition}\label{p weak-2-local and 2-local do not coincide} Let $X$ and $Y$ be Banach spaces with $Y$ infinite dimensional. Suppose $F$ is a proper norm-dense subspace of $Y$. Let $\mathcal{S}$ be the set of all finite rank operators $S$ in $L(X,Y)$ such that $S(X) \subset F$. Then the 2-local $\mathcal{S}$ maps are the $1$-homogeneous maps from $X$ to $Y$ whose image is contained in $F$, while the set of weak-2-local $\mathcal{S}$ maps is the set of all $1$-homogeneous maps from $X$ to $Y$.
\end{proposition}

\begin{proof} As claimed in the introduction of \cite{JorPe}, it is not hard to see that every $1$-homogeneous map $\Delta : X\to Y$ with $\Delta(X)\subset F$ is a 2-local $\mathcal{S}$ map. Clearly, every 2-local $\mathcal{S}$ map $\Delta: X\to Y$ is $1$ homogeneous (cf. \cite[Lemma 2.1]{CaPe2016}) and satisfies $\Delta (X) \subset F$.\smallskip

Suppose now that $\Delta : X\to Y$ is a $1$-homogeneous map. Pick $x,y\in X$ and $\phi\in Y^*$. We can assume $\phi\neq 0$, otherwise $\phi \Delta (x) = \phi S(x) = 0 =\phi \Delta (y) = \phi S(y)$, for every $S\in \mathcal{S}$. We deduce, from the norm-density of the subspace $F$ in $Y$ and the continuity of $\phi\neq 0$, that $\phi (F)= \mathbb{C}.$  We assume first that $y = \alpha x$ for certain $\alpha\in \mathbb{C}$.  Since $\Delta$ is $1$-homogeneous, we have $\Delta(y) = \alpha \Delta (x)$. Take $\varphi_x\in X^*$ satisfying $\varphi_x (x) = 1$. Since $\phi (F)= \mathbb{C},$ we can pick $m\in F$ such that $\phi (m) = \phi \Delta (x)$. The linear operator $S = m\otimes \varphi_x$ lies in $\mathcal{S}$ and satisfies $$\phi S(x) = \phi (m) = \phi \Delta (x), \hbox{ and }\phi S(y) = \alpha \phi S(x)  = \alpha \phi \Delta (x) = \phi \Delta (y).$$

We assume now that $x$ are linearly independent. Let us choose $\varphi_x, \varphi_y\in X^*$ satisfying $\varphi_x (x) = 1$, $\varphi_x (y) =0$, $\varphi_y (y) =1$ and $\varphi_y (x)=0$. Since $\phi (F) = \mathbb{C}$, we can choose $m,n\in F$ such that $\phi (m) = \phi \Delta(x)$ and $\phi (n) = \phi \Delta (y)$. The linear mapping $S = m\otimes \varphi_x + n\otimes \varphi_y$ belongs to $\mathcal{S}$ and clearly $\phi \Delta (x) = \phi S(x)$ and $\phi \Delta (y) = \phi S(y)$, which finishes the proof.
\end{proof}

In what follows we shall revisit some properties which are essentially inherent to derivations and weak-2-local derivations on C$^*$-algebras of continuous functions. We begin with the case a of a derivation on $C(\Omega, A),$ where $\Omega$ is a compact Hausdorff space and $A$ is a C$^*$-algebra. This case is probably part of the folklore in the theory of derivations, we include here an sketch of the arguments for completeness.\smallskip

Following the notation in \cite{JorPe}, given $t\in \Omega$, $\delta_{t} : C(\Omega,A) \to A$ will denote the $^*$-homomorphism defined by $\delta_t (X) = X(t)$. We observe that the space $C(\Omega,A)$ also is a Banach $A$-bimodule with products $(a X) (t) = a X(t)$ and $(X a) (t) = X(t) a$, for every $a\in A$, $X\in C(\Omega,A)$, and the mapping $\delta_{t} : C(\Omega,A) \to A$ is an $A$-module homomorphism.\smallskip

Henceforth, each element $a$ in $A$, the symbol $\Gamma (a)= 1\otimes a = \widehat{a}$ will denote the constant function with value $a$ from $\Omega$ into $A$. The mapping $\Gamma: A \to C(\Omega, A)= C(\Omega) \otimes A$, $a\mapsto \Gamma (a)$,  is an $A$-module homomorphism.\smallskip

We begin with a new technical lemma.

\begin{lemma}\label{l evaluation at positive functionals} Let $D: A\to A$ be a derivation on a C$^*$-algebra. Suppose $\varphi$ is a positive functional on $A$ and $a$ is an element in $A$ such that $\varphi (a a^* + a^* a) =0$. Then $\varphi D(a)=0$. Consequently, if $\Delta:A \to A$ is a weak-2-local derivation we have $\varphi \Delta (a) =0$ for every $a$ and $\varphi$ as above.
\end{lemma}

\begin{proof} Under the hypothesis of the Lemma, it follows from the Cauchy-Schwarz inequality that $|\varphi (a z)|^2 \leq \varphi(a a^*) \varphi(z^* z) =0$, and hence $\varphi (az )=0$, for every $z\in A$.  Similarly, we get $\varphi (a^* z)=0$, for every $z\in A$. So, if $p(\lambda,\mu)$ is a complex polynomial in two variables with zero constant term, we have $\varphi (p(a,a^*) z)=0$, for every $z\in A$. Let $B$ denote the C$^*$-subalgebra of $A$ generated by $a$ (we observe that $B$ need not be commutative). The continuity of $\varphi$ and the previous arguments show that $\varphi (bz )=0$, for every $b\in B$, $z\in A$. Similar ideas are valid to prove that $\varphi (z b)=0$, for every $b\in B$, $z\in A$.\smallskip

Since $B$ is a C$^*$-algebra, and hence it contains an approximate unit, we deduce from the Cohen factorization theorem (cf. \cite[Theorem VIII.32.22]{HewRoss}) the existence of $b,c\in B$ such that $a = bc$. The statement proved in the above paragraph shows that $$\varphi D(a) = \varphi (D(b) c + b D(c)) =0,$$ as we wanted. The rest is clear.
\end{proof}

Let $L$ denote a locally compact Hausdorff space, and let $A$ be a C$^*$-algebra. According to the terminology introduced above, given $t\in L$, $\delta_{t} : C_0(L,A) \to A$ will denote the $^*$-homomorphism defined by $\delta_t (X) = X(t)$. The space $C_0(L,A)$ also is a Banach $A$-bimodule with products $(a X) (t) = a X(t)$ and $(X a) (t) = X(t) a$, for every $a\in A$, $X\in C_0(L,A)$, and the mapping $\delta_{t} : C_0(L,A) \to A$ is an $A$-module homomorphism.\smallskip

Let $K$ and $O$ be subsets of $L$ with $K\subseteq O\neq L$, $K$ compact, and $O$ open. Pick, via Urysohn's lemma, a continuous function $f_{K,O}\in C_0(L)$ with $0\leq f_{K,O}\leq 1$, $f_{K,O}|_{K}\equiv 1$ and $f_{K,O}|_{L\backslash O} \equiv 0$. For each $a$ in $A$, the symbol $\Gamma_{f_{K,O}} (a)= f_{K,O}\otimes a $ will denote the continuous function on $L$ mapping each $t$ to $f_{K,O}(t) a$. The mapping $\Gamma_{{f_{K,O}}}: A \to C_0(L, A)= C_0(L) \otimes A$, $a\mapsto\Gamma_{{f_{K,O}}} (a)$,  is an $A$-module homomorphism.\smallskip

Suppose $\Delta:C_0(L,A)\to C_0(L,A)$ is a weak-2-local derivation. Let $Y$ be an element in $C_0(L, A)$ satisfying $Y(t_0)=0$ for some $t_0\in L$. Let us fix a positive functional $\varphi\in A^*$ and consider the positive functional $\varphi\otimes \delta_{t_0}$ on $C_0(L,A)$ defined by $(\varphi\otimes \delta_{t_0}) (X) = \varphi (X (t_0))$ ($X\in C_0(L,A)$). Clearly $(\varphi\otimes \delta_{t_0}) (Y^* Y + Y Y^*) =0$. Lemma \ref{l evaluation at positive functionals} affirms that $(\varphi\otimes \delta_{t_0}) \Delta (Y) = 0,$ and since $\varphi$ was arbitrarily chosen, we have $\Delta (Y) (t_0)=0.$ We gather this information in the next lemma.

\begin{lemma}\label{l weak-2-local derivation anihilates on deltat} Let $\Delta:C_0(L,A)\to C_0(L,A)$ be a weak-2-local derivation. Suppose $Y$ is an element in $C_0(L, A)$ satisfying $Y(t_0)=0$ for some $t_0\in L$. Then $\Delta (Y) (t_0)=0.$ In particular, if $K$ and $O$ are subsets of $L$ with $K\subseteq O\neq L$, $K$ compact, and $O$ open, and $f_{K,O}\in C_0(L)$ with $0\leq f_{K,O}\leq 1$, $f_{K,O}|_{K}\equiv 1$ and $f_{K,O}|_{L\backslash O} \equiv 0$, we have $$ \Delta (X) (t) =\delta_t \Delta \Gamma_{{f_{K,O}}} \delta_t (X),$$ for every $t\in K$, and every $X\in C_0(L,A)$.
\end{lemma}

\begin{proof} The first statement has been proved in the comments preceding the lemma. Fix an arbitrary functional $\varphi\in A^*$, $X$, $f_{K,O}\in C_0(L)$ and $t$ as in the hypothesis. By assumptions, there exists a derivation $D$ on $C_0(L,A),$ depending on $(\varphi\otimes \delta_{t})$, $X$ and $\Gamma_{{f_{K,O}}} (X(t)),$ such that $$\varphi (\Delta (X) (t)) = (\varphi\otimes \delta_{t}) \Delta (X) = (\varphi\otimes \delta_{t}) D(X) = \varphi (D(X) (t)),$$ and $$\varphi (\Delta (\Gamma_{{f_{K,O}}} (X(t))) (t)) =  \varphi (D(\Gamma_{{f_{K,O}}} (X(t))) (t)).$$ By the first statement, which is of course valid for derivations, we have $$D (X) (t) =D (\Gamma_{{f_{K,O}}} (X(t))) (t),$$ and hence $\varphi (\Delta (X) (t) -\Delta (\Gamma_{{f_{K,O}}} (X(t))) (t))=0$. Finally, the Hahn-Banach theorem gives the desired statement.
\end{proof}

When $\Omega$ is a compact Hausdorff space. We can always replace in Lemma \ref{l weak-2-local derivation anihilates on deltat} the function $\Gamma_{{f_{K,O}}}$ with $\Gamma$ to prove that for every weak-2-local derivation $\Delta$ on $C(\Omega,A)$ we have $$ \Delta (X) (t) =\delta_t \Delta \Gamma \delta_t (X) ,$$ for every $t\in \Omega$, $X\in C_0(L,A)$. As a consequence, hypothesis $(b)$ in \cite[Theorem 2.13]{JorPe} can be relaxed. We can state now a generalization of the just quoted theorem for continuous functions on a locally compact Hausdorff space.

\begin{theorem}\label{t pattern} Let $L$ be a locally compact Hausdorff space. Suppose $A$ is a C$^*$-algebra such that every weak-2-local derivation on $A$ is a {\rm(}linear{\rm)} derivation. Then every weak-2-local derivation on $C_0(L,A)$ is a {\rm(}linear{\rm)} derivation.
\end{theorem}

\begin{proof} Let $\Delta : C_0(L,A) \to C_0(L,A)$ be a weak-2-local derivation. Having in mind Theorem 3.4 in \cite{EssaPeRa16} (see also \cite{EssaPeRa16b}) we deduce that it is enough to prove that $\Delta$ is linear. We shall show that $\delta_t \Delta$ is linear for every $t\in L$.\smallskip

Let us pick $t_0\in L$. Let $K$ and $O$ be arbitrary subsets of $L$ with $K\subseteq O\neq L$, $K$ compact, $t_0\in K$, and $O$ open, and let $f_{K,O}\in C_0(L)$ be an arbitrary function with $0\leq f_{K,O}\leq 1$, $f_{K,O}|_{K}\equiv 1$ and $f_{K,O}|_{L\backslash O} \equiv 0$. Lemma \ref{l weak-2-local derivation anihilates on deltat} assures that $$ \Delta (X) (t) =\Delta (\Gamma_{{f_{K,O}}} (X(t))) (t),$$ for every $t\in K$, $X\in C_0(L,A)$. Fix $K$, $O$ and $f_{K,O}$ satisfying the above properties.\smallskip

The mapping $\delta_{t_0} : C_0(L,A) \to A$ is a continuous $A$-module homomorphism. Unfortunately, the operator $\Gamma_{{f_{K,O}}}: A\to C_0(L,A)$ need not be a homomorphism, however, it satisfies $\Gamma_{{f_{K,O}}} (a) A = a A$ and $A \Gamma_{{f_{K,O}}} (a) = A a$, for every $a\in A$, and every $A\in C_0(L,A)$. We cannot apply Lemma 2.3$(c)$ in \cite{JorPe} to prove that $\delta_{t_0} \Delta \Gamma_{{f_{K,O}}} : A\to A$ is a weak-2-local derivation. In this case, a more elaborated argument is needed. We shall show next that \begin{equation}\label{eq weak2der with Gammafko} \delta_{t_0} \Delta \Gamma_{{f_{K,O}}} : A\to A
 \end{equation}is a weak-2-local derivation. For this purpose, we pick a derivation $D $ on $C_0(L,A)$ and we shall show that $\delta_{t_0} D \Gamma_{{f_{K,O}}}$ is a derivation. To simplify notation let $g_{K,O} = f_{K,O}^{\frac12}\in C_0(L)$. Given $a,b\in A$, Lemma \ref{l weak-2-local derivation anihilates on deltat} implies that $$ \delta_{t_0} D \Gamma_{{f_{K,O}}} ( a b) = \delta_{t_0} D ({f_{K,O}}\otimes (a b)) =  \delta_{t_0} D (({g_{K,O}}\otimes a ) ({g_{K,O}}\otimes b )) $$ $$= \delta_{t_0} ( D (({g_{K,O}}\otimes a ))  ({g_{K,O}}\otimes b )) + \delta_{t_0} (({g_{K,O}}\otimes a ) D (({g_{K,O}}\otimes b )))$$ $$= \delta_{t_0} ( D (({g_{K,O}}\otimes a )))\ b + a\ \delta_{t_0} ( D (({g_{K,O}}\otimes b )))$$ $$ = \delta_{t_0} ( D (({f_{K,O}}\otimes a )))\ b + a\ \delta_{t_0} ( D (({f_{K,O}}\otimes b )))= (\delta_{t_0} D \Gamma_{{f_{K,O}}}) (a)\ b + a\ (\delta_{t_0} D \Gamma_{{f_{K,O}}}) (b),$$ which proves what we desired.\smallskip

To prove \eqref{eq weak2der with Gammafko} we observe that given $a,b\in A$, and a functional $\varphi\in A^*$, by the hypothesis on $\Delta$ there exists a derivation $D$ on $C_0(L,A)$, depending on $\Gamma_{{f_{K,O}}} (a)$, $\Gamma_{{f_{K,O}}} (b)$, and the functional $\varphi\otimes \delta_{t_0}\in C_0(L,A)^*$, such that $$\varphi (\delta_{t_0} \Delta  \Gamma_{{f_{K,O}}}) (a)= (\varphi\otimes \delta_{t_0}) \Delta  \Gamma_{{f_{K,O}}} (a) = (\varphi\otimes \delta_{t_0}) D  \Gamma_{{f_{K,O}}} (a)= \varphi (\delta_{t_0} D  \Gamma_{{f_{K,O}}}) (a)$$ and $$\varphi (\delta_{t_0} \Delta  \Gamma_{{f_{K,O}}}) (b)= (\varphi\otimes \delta_{t_0}) \Delta  \Gamma_{{f_{K,O}}} (b) = (\varphi\otimes \delta_{t_0}) D  \Gamma_{{f_{K,O}}} (b)= \varphi (\delta_{t_0} D  \Gamma_{{f_{K,O}}}) (b),$$ witnessing the desired conclusion because $\delta_{t_0} D  \Gamma_{{f_{K,O}}}$ is a derivation on $A$.\smallskip

By the hypothesis on $A$ and \eqref{eq weak2der with Gammafko}, $\delta_{t_0} \Delta \Gamma_{{f_{K,O}}} $ is a linear derivation, and thus $$ \delta_{t_0} \Delta \Gamma_{{f_{K,O}}} (X(t_0)+Y(t_0)) = \delta_{t_0} \Delta \Gamma_{{f_{K,O}}} (X(t_0)) +\delta_{t_0} \Delta \Gamma_{{f_{K,O}}} (Y(t_0)),$$ for every $X,Y$ in $C_0(L,A)$. Lemma \ref{l weak-2-local derivation anihilates on deltat} implies that $$\delta_{t_0} \Delta (X+Y) = \delta_{t_0} \Delta \Gamma_{{f_{K,O}}} (X(t_0)+Y(t_0))  $$ $$= \delta_{t_0} \Delta \Gamma_{{f_{K,O}}} (X(t_0)) +\delta_{t_0} \Delta \Gamma_{{f_{K,O}}} (Y(t_0))= \delta_{t_0} \Delta (X)+ \delta_{t_0} \Delta (Y),$$ for every $X,Y\in C_0(L,A)$. It follows from the arbitrariness of $t_0$ that $\Delta$ is linear.\end{proof}

Corollaries 3.5 and 3.6 in \cite{CaPe2016} assert that every weak-2-local derivation on a dual von Neumann algebra or on a compact C$^*$-algebra is a linear derivation. Combining these results with Theorem \ref{t pattern} we deduce our main result on weak-2-local derivations on $C_0(L,B(H))$.

\begin{theorem}\label{t weak2local derivations C0LB(H)} Let $M$ be an atomic von Neumann algebra or a compact C$^*$-algebra, and let $L$ a locally compact Hausdorff space. Then every weak-2-local derivation on $C_0(L,M)$ is a linear derivation. In particular, for every complex Hilbert space $H$, every weak-2-local derivation on $C_0 (L,B(H))$ or on  $C_0(L,K(H))$ is a linear derivation.$\hfill\Box$
\end{theorem}

The proofs of Theorem \ref{t pattern} above remains valid to obtain the following result.

\begin{theorem}\label{t pattern 2local} Let $L$ be a locally compact Hausdorff space. Suppose $A$ is a C$^*$-algebra such that every 2-local derivation on $A$ is a {\rm(}linear{\rm)} derivation. Then every 2-local derivation on $C_0(L,A)$ is a {\rm(}linear{\rm)} derivation.$\hfill\Box$
\end{theorem}

\begin{corollary}\label{c 2local derivations C0LB(H)} Let $M$ be a von Neumann algebra, and let $L$ a locally compact Hausdorff space. Then every 2-local derivation on $C_0(L,M)$ is a linear derivation. $\hfill\Box$
\end{corollary}

After proving that every weak-2-local derivation on $C_0 (L,B(H))$ (respectively, on $C_0(L,K(H))$) is a linear derivation, it becomes more interesting to find concrete representations for derivations on these spaces without requiring elements in their second duals. This will be done in the next sections.\smallskip

Before ending this section we recall a well known, and standard, argument linking derivations and $^*$-derivations. Let $A$ and $B$ be C$^*$-algebras.  For each mapping $\Psi : A\to B$, we define a mapping $\Psi^{\sharp}: A \to B$ given by $\Psi^{\sharp} (a) =\Psi(a^*)^*$ ($a\in A$). Clearly, $\Psi$ is linear if and only if $\Psi^{\sharp}$ is. A mapping satisfying $\Psi^{\sharp} =\Psi$ is called \emph{symmetric}. It is proved in \cite{CaPe2015} that every weak-2-local $^*$-derivation on a C$^*$-algebra and a derivation. Every mapping $\Delta$ on a C$^*$-algebra $A$ can be written as a linear combination $\Delta=\Delta_1 + i \Delta_2$ of two symmetric maps $\Delta_1 =\frac12 (\Delta+\Delta^{\sharp})$ and $\Delta_2 =\frac{1}{2 i}  (\Delta-\Delta^{\sharp})$. Notice that if $\Psi$ is symmetric and a weak 2 local derivation we do not know a priori that $\Psi$ it is a weak 2 local $^*$ derivation. Hence, it is an open problem whether every weak-2-local derivation which is also a symmetric mapping is linear or not (compare \cite{NiPe2014,NiPe2015,CaPe2015} and \cite[Problem 1.4]{CaPe2016}).\smallskip

Throughout this note, the centre of a C$^*$-algebra $A$ will be denoted by $Z(A)$. Let $a,b$ be elements in a C$^*$-algebra $A$.

\section{Inner derivations determined by elements in the multiplier algebra}\label{sec:norm estimations}

It is  well known that elements $a$ and $b$ in a C$^*$-algebra $A$ define the same inner derivation on $A$ (i.e. $\hbox{ad}_{a} = \hbox{ad}_{b}$) if and only if $a-b\in Z(A)$. A simple application of the triangular inequality shows that the norm of the inner derivation $\hbox{ad}_{a}$ is bounded by the double of the distance from $a$ to $Z(A)$, that is, \begin{equation}\label{eq distance to the centre bounds the norm of an inner der} \| \hbox{ad}_{a}\|= \left\| [a,. \ ] \right\| \leq 2 \dist (a, Z(A)).
 \end{equation}
The question whether the inequality in \eqref{eq distance to the centre bounds the norm of an inner der} is in fact an equality has been the goal of study of many researchers (compare \cite{KadLanRing67,Stampfli,Gaje72,Szi73,Ell78,Arch78,Somm93,Somm94}, and \cite{ArchSom2004}, among many others). In \cite[Example 6.2]{KadLanRing67} the authors exhibit a unital C$^*$-algebra $\mathcal{U}$ containing a sequence of unitary elements $(u_n)\subset \mathcal{U}$ such that $\displaystyle \| \hbox{ad}_{u_n}\|= \left\| [u_n,. \ ] \right\| \stackrel{n\to \infty}{\longrightarrow} 0,$ and $\dist (u_n, Z(\mathcal{U})) =1$, for every $n$. In other words, the inequality in \eqref{eq distance to the centre bounds the norm of an inner der} can be strict. To deal with a more detailed study, R.J. Archbold introduce in \cite{Arch78} the following constants. Given a C$^*$-algebra $A$,  $\mathcal{K}(A)$ (respectively, $\mathcal{K}_{s}(A)$) will denote
be the smallest number $\mathcal{K}$ in $[0, +\infty]$ such that $$\dist (a, Z(A)) \leq \mathcal{K} \| \hbox{ad}_{a}\|, \hbox{ for all  } a\in  A$$ (respectively, for all $a = a^*$ in $A$). It is known that $\frac 12 \mathcal{K}(A) \leq \mathcal{K}_{s} (A) \leq \mathcal{K}(A)$. Despite the wide number of papers studying these constants, we shall highlight the results exhibited next. We have already commented an example due to R.V. Kadison, E.C. Lance and J.R. Ringrose of a unital C$^*$-algebra $A$ satisfying $\mathcal{K}(A) =\infty$ (cf. \cite[Example 6.2]{KadLanRing67}), another example can be found in \cite{Kyle77}. It is proved in \cite[Theorem 3.1]{KadLanRing67} that $\mathcal{K}_{s} (A)\leq \frac12$ when $A$ is a von Neumann algebra, and in the general setting, $\mathcal{K} (A)<\infty $ if and only if the space of all inner derivations on $A$ is closed in the Banach space of all derivations on $A$. Clearly  $\mathcal{K} (A)=0$ when $A$ is commutative. It is known that  $\mathcal{K} (A)\leq \frac12$ whenever $A$ is unital and primitive \cite{Stampfli}, or just prime \cite[Corollary 2.9]{Somm94}, or a von Neumann algebra \cite{Szi73}, or an AW$^*$-algebra \cite{Ell78}. For an arbitrary unital, non-commutative C$^*$-algebra $A$ either $\mathcal{K} (A) =\frac12$, or $\mathcal{K} (A) = \frac{1}{\sqrt{3}}$, or $\mathcal{K} (A)\geq 1$, depending on the topological properties of the primitive and primal ideals of $A$ (see \cite{Somm94}).\smallskip

Let $M$ be a von Neumann algebra, and let $B\cong C(\Omega)$ be a unital abelian C$^*$-algebra. In \cite[Theorem 2.3]{AkJohn79}, C.A. Akemann and B.E. Johnson establish that every derivation of the C$^*$-tensor product $B \otimes M\cong C(\Omega, M)$ is inner, that is, for each derivation $D: C(\Omega, M) \to C(\Omega, M)$, there exists $Z_0\in C(\Omega,M)$ satisfying $D(X)= [Z_0, X]$, for every $X\in C(\Omega,M)$. We shall revisit the original proof of Akemann and Johnson by combining it with results due to R.V. Kadison, E.C. Lance, J.R. Ringrose \cite{KadLanRing67}, to strengthen the conclusion.\smallskip

Let us recall that for every C$^*$-algebra, $A$, the \emph{multiplier algebra} of $A$, $M(A)$, is the set of all elements $x\in A^{**}$ such that $Ax, xA \subseteq A$. We notice that $M(A)$ is a C$^*$-algebra and contains the unit element of $A^{**}$. Clearly, $M(A)=A$ whenever $A$ is unital, otherwise $M(A)$ is an extension of $A$ contained in $A^{**}$. For a locally compact Hausdorff space $L$, the multiplier algebra of $C_0(L)$ is isomorphic to the algebra $C_b(L)$ of all bounded continuous functions on $L$; hence its spectrum is homeomorphic to the Stone-\v{C}ech compactification of $L$ (see \cite[Theorem III.6.30]{Tak}).\smallskip

We have already recalled that Sakai's theorem affirms that every derivation on a von Neumann algebra is inner. Another contribution due to Sakai shows that every derivation of a simple C$^*$-algebra with unit is inner (see \cite[Theorem 4.1.11]{S}). In order to deal with a simple C$^*$-algebra $A$ without unit, Sakai introduced in \cite{Sak71a} the multiplier algebra $M(A)$ (originally called the ``derived C$^*$-algebra'' of $A$ by Sakai). It is further established in the same paper that every derivation on $A$ extends to a unique derivation on $M(A)$, and that every derivation of $M(A)$ is inner. In particular, for each derivation $D$ on a simple C$^*$ algebra $A$ there exists $m\in M(A)$ such that $D(a)=[m,a]$.\smallskip

For separable C$^*$-algebras, C.A. Akemann, G.E. Elliott, G.K. Pedersen and J. Tokiyama \cite{AkEllPedTomi76} and \cite{Ell77} characterized those separable C$^*$-algebras $A$ satisfying that every derivation on $A$ is inner in $M(A)$. These are precisely the C$^*$-algebras which are the C$^*$-algebra direct sum of a family of simple C$^*$-algebras and a full C$^*$-algebra (that is, a C$^*$-algebra with only trivial central sequences).\smallskip

Additional extensions of a C$^*$-algebra $A$, like the local multiplier algebra $M_{loc}(A)$, have been introduced with the aim of proving that every derivation on $A$ extends to an inner derivation on the corresponding extension. G.K. Pedersen proved that for a separable C$^*$-algebra $A$ every derivation on $A$ extends to an inner derivation of $M_{loc}(A)$ (see \cite{Ped1978}). There are some recent extensions of these results due to D.W. Somerset \cite{Somm00} and I. Gogi{\'c} \cite{gogic1,gogic2}. But  much less seems to be known about proper extensions of Sakai, i.e. examples of algebras where the derivations are inner in $M(A)$. Let us note that $M(A)=A$ when $A$ is unital. So, being inner in $M(A)$ seems to be a natural extension for derivations in the setting of non unital C$^*$-algebras\smallskip

Given a C$^*$ algebra $A$ and a locally compact Hausdorff space $L$, the multiplier algebra $M(C_0(L,A))=M(C_0(L)\otimes A)$ coincides with the C$^*$-algebra $C_b( L, (M(A),\tau_s))$ of all bounded and continuous functions from $L$ into the space $(M(A),\tau_s)$, where $\tau_s$ stands for the \emph{strict topology} on $M(A)$ \cite[Corollary 3.4]{APT}. If the C$^*$-algebra $A$ has a unit then $M(A) = A$ and the strict topology coincides with the norm topology. If $A=K(H)$ then the strict topology of $M(A)=B(H)$ is the strong$^*$ topology of $B(H)$ (denoted by $s^* (B(H),B(H)_*)$ or simply by $s^*$). We refer to \cite[Definition 1.8.7]{S} for the concrete definition of the strong$^*$ topology. Thus $M(C_0(L,K(H)))=C_b(L,(B(H),s^*))$ (\cite[Corollary 3.4]{APT}). Once we have this identification, our purpose is to show that if $L$ is locally compact and paracompact then every derivation $D$ on $C_0(L,M)$ is inner in its multipler algebra, and for a general locally compact space $L$ the same happens for derivations on $C_0(L,K(H))$ and on $C_0(L,B(H))$. We also get bounds of the norms of the elements in the multiplier algebras representing the derivations.

\subsection{Derivations on $C_0(L,A)$ with $L$ paracompact}\ \smallskip

In most of the positive results representing a derivation $D$ on a C$^*$-algebra $A$ as an inner derivation of the form $\hbox{ad}_{z}$, with $z\in A$ or $z\in M(A)$, there exists a link between $\|D\|$ and $\|z\|$. This link does not appear in \cite[Theorem 2.3]{AkJohn79}, where C.A. Akemann and B.E. Johnson establish that, if $\Omega$ is a Hausdorff compact space, then every derivation of the C$^*$-tensor product $C(\Omega) \otimes M\cong C(\Omega, M)$ is inner, that is, for each derivation $D: C(\Omega, M) \to C(\Omega, M)$, there exists $Z_0\in C(\Omega,M)$ satisfying $D(X)= [Z_0, X]$, for every $X\in C(\Omega,M)$. We shall revisit the original proof of Akemann and Johnson by combining it with results due to R.V. Kadison, E.C. Lance, J.R. Ringrose \cite{KadLanRing67}, to extend the result to the non unital case.

\begin{theorem}\label{t AkJohn paracompact strengtened with control on the norm} Let $M$ be a von Neumann algebra, let $L$ be a locally compact space which in addition is paracompact.  Let $\varepsilon$ be a positive element. Then for each $^*$-derivation $D: C(L, M)\to C(L, M)$, there exists $Z_0\in C_b(L,M)$ such that $D(X)= [Z_0, X]$, for every $X\in C_0(L,M)$, $Z_0^*=-Z_0$, and $$\|Z_0 \| \leq (\mathcal{K}_s (M)+\varepsilon)\|D\| \leq \frac{1+2 \varepsilon}{2} \|D\|.$$ Consequently, for each derivation $D: C_0(L, M)\to C_0(L, M)$, there exists $Z_0\in C_b(L,M)$ such that $D(X)= [Z_0, X]$, for every $X\in C_0(\Omega,M)$ and $\|Z_0 \| \leq  (1+2 \varepsilon)\|D\|$.
\end{theorem}

\begin{proof} The proof of the statement will follow by an adaptation of the original arguments in \cite[Theorem 2.3]{AkJohn79} with a slight modification motivated by the conclusions in \cite[Theorem 3.1]{KadLanRing67}.\smallskip

Let $D: C_0(L, M)\to C_0(L, M)$ be a $^*$-derivation. Henceforth, $\tau$ will stand for the topology of $L$. According to the terminology employed in this note, let $t\in L$ and let $t\in K\subsetneq O\subsetneq L$, with $K$ compact and $O$ open. Let us write $\Gamma =\Gamma_{f_{K,O}}.$ It follows from Lemma \ref{l weak-2-local derivation anihilates on deltat} that, for each $t\in L$, the mapping $\delta_t D \Gamma : M\to M$ is a $^*$-derivation on $M$ (we further know that the definition does not depend on $K$ and $O$). Fix $t_0\in L$ and a compact neighbourhood $K$ of $t_0$. It also holds that the mapping $\Upsilon_K : K \to \hbox{Der} (M),$ $t\mapsto \Upsilon(t)=\delta_t D \Gamma$ is a $\tau$-to-(point-norm) continuous mapping (here $\Gamma=\Gamma_{f_{K,O}}$ is fixed). Therefore $\Upsilon_K (K)$ is a point-norm compact subset in Der$(M)$. Theorem 2.1 in \cite{AkJohn79} assures that $\Upsilon_K (K)$ is norm compact, and hence the point-norm and the norm topologies coincide on $\Upsilon_K (K)$. Thus
$\Upsilon_K$ is $\tau$-to-point norm continuous. Define now a mapping $\Upsilon: L\to \hbox{Der}(M)$ given by $\Upsilon_K(t)$ if $t\in K$. $\Upsilon$ is well defined by Lemma \ref{l weak-2-local derivation anihilates on deltat} and it is $\tau$-to norm continuous by the local compactness of $L$ and the continuity of each $\Upsilon_K$.
\smallskip

By Sakai's theorem (see \cite{Sak60}), every derivation on $M$ is inner and the mapping $\theta: i M_{sa}/Z(M_{sa}) \to \hbox{$^*$-Der} (M)$, $a+Z(A)\mapsto \hbox{ad}_a$ is an isomorphism of Banach spaces. Therefore, the mapping $\theta^{-1} \Upsilon: \Omega \to i M_{sa}/Z(M_{sa})$ is continuous. Let $2^{i M_{sa}}$ denote the family of non-empty subsets of $i M_{sa}$. We define a carrier, $G: \Omega \to 2^{i M_{sa}}$, given by $$G(t) = (\theta^{-1} \Upsilon) (t) \cap B_{M} (0,(\mathcal{K}_s (M)+\varepsilon)\|D\|),$$ where $B_{M} (0,(\mathcal{K}_s (M)+\varepsilon)\|D\|)$ denotes the open unit ball in $M$ with center zero and radius $(\mathcal{K}_s (M)+\varepsilon)\|D\|)$, that is $G(t)$ is the set of all elements $c\in i M_{sa}$ belonging to the class $(\theta^{-1} \Upsilon) (t)\subset i M_{sa}/Z(M_{sa})$ with norm $\|c \|< (\mathcal{K}_s (M)+\varepsilon)\|D\|.$ By definition, $G(t)$ is convex and non-empty for every $t\in L$. It is not hard to check that $G$ is lower semi-continuous (compare \cite[examples in page 362]{Michael56}). It follows from \cite[Proposition 2.3]{Michael56} that $t\mapsto \overline{G(t)}$ is lower semi-continuous; and clearly $\overline{G(t)}$ is non-empty, closed and convex. Therefore, by Michael's selection principle (see \cite[Theorem 3.2]{Michael56}) $\overline{G}$ admits a continuous selection, that is, there exists a continuous function $Z_0: L \to i M_{sa}$ satisfying $Z_0 (t) \in \overline{G(t)}$ for every $t\in L$, and $\|Z_0 \| \leq (\mathcal{K}_s (M)+\varepsilon)\|D\|$. $Z_0$ is obviously a bounded function, and a glance to Lemma \ref{l weak-2-local derivation anihilates on deltat} is enough to convince ourself that $D = [Z_0,.]$, as we desired.\smallskip

By \cite[Theorem 3.1]{KadLanRing67} we know that $\mathcal{K}_{s} (M)\leq \frac12$. The final statement follows from the fact that every derivation can be written as a linear combination of two $^*$-derivations.
\end{proof}

Since each compact space also is paracompact, as an immediate corollary of the above theorem we get an straightened version of \cite[Theorem 2.3]{AkJohn79}, where the conclusion also contains an estimate of the norm on the function representing the derivation.

\begin{theorem}\label{t AkJohn strengtened with control on the norm} Let $M$ be a von Neumann algebra, and let $C(\Omega)$ be a unital abelian C$^*$-algebra. Let $\varepsilon$ be a positive element. Then for each $^*$-derivation $D: C(\Omega, M)\to C(\Omega, M)$, there exists $Z_0\in C(\Omega,M)$ such that $D(X)= [Z_0, X]$, for every $X\in C(\Omega,M)$, $Z_0^*=-Z_0$, and $$\|Z_0 \| \leq (\mathcal{K}_s (M)+\varepsilon)\|D\| \leq \frac{1+2 \varepsilon}{2} \|D\|.$$ Consequently, for each derivation $D: C(\Omega, M)\to C(\Omega, M)$, there exists $Z_0\in C(\Omega,M)$ such that $D(X)= [Z_0, X]$, for every $X\in C(\Omega,M)$ and $\|Z_0 \| \leq  (1+2 \varepsilon)\|D\|$.
\end{theorem}

\subsection{Derivations on $C_0(L,A)$ for a general locally compact Hausdorff space $L$}\ \smallskip

Let $\Omega$ be a compact Hausdorff space. We observe that $M_n (C(\Omega)) $ $= C(\Omega,M_n)$ $= C(\Omega)\otimes M_n$, and consequently, every derivation on $M_n(C(\Omega))$ is inner (see \cite[Theorem 2.3]{AkJohn79}). Our next result is an appropriate non-unital version of this fact.

\begin{lemma}\label{l C0(LM_n) has the inner derivation property} Let $A$ be a C$^*$-algebra. Then every derivation $D$ on $M_n (A)$ writes in the form $D(X) = [Z,X] $, where $Z=(z_{i,j})$ satisfies $z_{ij}\in M(A)\subseteq A^{**}$ for every $i\neq j$.
When $A=C_0(L)$ is a commutative C$^*$-algebra, we can also assume that $z_{ii}\in M(A)=C_b(L)$, for every $i$, or equivalently, $Z\in M_n (M(A))=M_n (C_b (L))$.
\end{lemma}

\begin{proof} Suppose $D: M_n (A)\to M_n (A)$ is a derivation. We can, obviously, identify $M_n (A)^{**}$ with $M_n (A^{**})$ in a canonical way. Since $D^{**} : M_n (A)^{**}\to M_n (A)^{**}$ is a derivation on a von Neumann algebra, by Sakai's theorem, there exists $Z= (z_{ij}) \in  M_n (A^{**})$ satisfying $D^{**}(X) = [Z,X]$ for every $X\in M_n (A^{**})$. We further know that $D (X) = [Z,X]\in M_n (A)$ for every $X\in M_n (A)$.\smallskip

Fix $i,j\in \{1,\ldots, n\}$ and $a\in A$. Since $$ M_n (A)\ni D (a\otimes {e}_{ij}) = [Z,a\otimes {e}_{ij}] = Z (a\otimes {e}_{ij}) - (a\otimes {e}_{ij}) Z $$ $$=(z_{ii} a - a z_{jj}) \otimes {e}_{ij} + \sum_{k=1, k\ne i}^n (z_{ki} a) \otimes {e}_{kj} - \sum_{k=1, k\neq j}^n (a z_{jk}) \otimes {e}_{ik},$$ we deduce that $(z_{ii} a - a z_{jj})$, $z_{ki} a$, and $a z_{jk}$ all lie in $A$, for every $a\in A$, $k\neq i,j$ and every $i,j\in \{1,\ldots, n\}$. Therefore, $z_{ki}, z_{jk}\in M(A)$ for every $k\neq i,j$ and every $i,j\in \{1,\ldots, n\}$. This proves the first statement.\smallskip

Suppose now, that $A= C_0(L)$ is a commutative C$^*$-algebra. It follows from the above identities that $z_{ii} - z_{jj}\in M(A)$, for every $i,j\in \{1,\ldots,n\}$. Let  $D : M_n(A) \to M_n (A)$ be a derivation, and  let $Z$ be an element in $M_n (A^{**})$ satisfying $D^{**} (X) = [Z,X]$ for every $X\in M_n(A^{**})$. In this case the element $z_{nn}\otimes I_n\in Z(M_n (A^{**}))$, the center of  $M_n (A^{**})= A^{**}\otimes M_n $, and hence $$D^{**} (X) = [Z,X] = [Z-z_{nn} I_n,X], $$ for every $X\in M_n (A)^{**}$. Replacing $Z\in M_n (A)^{**}$ with $W= Z-z_{nn} I_n\in M_n (A)^{**}$ we can assume that $z_{nn}=0$. Applying the first statement to $W= Z-z_{nn} I_n$ we deduce that $w_{ij}\in M(A)$ for every $i,j\in \{1,\ldots, n\}$ and $D(X) = [W,X]$, for every $X\in M_n(A)$, which finishes the proof.
\end{proof}

We shall improve the result in Lemma \ref{l C0(LM_n) has the inner derivation property} with an appropriate control on the norm of the function $Z_0$ appearing in the statement in terms of the norm of the represented derivation $D$. We are led to the following extension of \cite[Proposition 2.8]{JorPe}.

\begin{lemma}\label{l C0(LM_n) has the idp control of the norm} Let $L$ be a locally compact Hausdorff space. Then for each $^*$-derivation $D: C_0(L,M_n)\to C_0(L,M_n)$ there exists a continuous and bounded function $Z_0: L\to M_n$ satisfying $\|Z_0\|_{\infty}\leq  \|D\|$, $Z_0^* (t) =- Z_0 (t)$, for every $t\in L$,  and $D(X) =[Z_0,X]$, for every $X\in C_0(L,M_n)$.
\end{lemma}

\begin{proof} Find, via Lemma \ref{l C0(LM_n) has the inner derivation property}, a bounded continuous function $Z_1: L\to M_n$ satisfying $Z_1^* (t) =- Z_1 (t)$, for every $t\in L$,  and $D(X) =[Z_1,X]$, for every $X\in C_0(L,M_n)$. Since $Z_1$ is bounded, we have $\|Z_1\|_{\infty} <\infty$. Pick $t\in L$ and an open set $O\subsetneq L$ such that $t\in O$, and keep in mind the notation employed before. The mapping $\delta_t D \Gamma_{f_{\{t\},O}} : M_n \to M_n$ is a $^*$-derivation. Clearly $\|\delta_t D \Gamma_{f_{\{t\},O}}\|\leq \|D\|$ and $\delta_t D \Gamma_{f_{\{t\},O}} (a) =[Z_1(t),a]$ for every $a\in M_n$. Since, by \cite[Corollary 1]{Stampfli} $\|[Z_1(t),.]\| = \diam(\sigma(Z_1(t)))$ (let us observe that $Z_1(t)^*=-Z_1(t)$), we deduce that $\diam(\sigma(Z_1(t)))\leq \|D\|$, for every $t\in L$. It is also obvious that $\sigma(Z_1(t))$ is a finite subset of $i \mathbb{R}$.\smallskip

The arguments in the proof of \cite[Proposition 2.8]{JorPe} show that the function $\sigma_{min} : L\to \mathbb{C}$, $\sigma_{min} (t) := \lambda\in \sigma(Z_1(t))$, where $\lambda$ is the unique element in $\sigma(Z_1(t))\subseteq i \mathbb{R}$ satisfying $|\lambda| = \min\{|\mu|: \mu \in \sigma (Z_1(t))\}$, is continuous (and bounded under our assumptions). The mapping $\sigma_{min}\otimes I_n : L \to M_n$ is center valued, bounded and continuous. We further know that $0\in \sigma ((Z_1-\sigma_{min}\otimes I_n) (t)) = \sigma_{min} (t) + \sigma(Z_1(t)) \subseteq i \mathbb{R}^{+}_0$,  and $$\|D\|\geq \diam(\sigma(Z_1(t))) = \diam(\sigma((Z_1-\sigma_{min}\otimes I_n) (t))),$$ for every $t\in L$. The proof concludes by taking $Z_0 = Z_1-\sigma_{min}\otimes I_n$.
\end{proof}

The next natural step is to consider a more general version of Lemma \ref{l C0(LM_n) has the inner derivation property} by replacing $M_n$ with a more general von Neumann algebra.\smallskip

We can state now a first ``local'' representation theorem for derivations on $C_0(L,M)$, where $M$ is an arbitrary von Neumann algebra and $L$ is a locally compact Hausdorff space.

\begin{theorem}\label{t representation of derivations on C0LM} Let $M$ be a von Neumann algebra, let $L$ be a locally compact Hausdorff space, and let $D: C_0(L,M) \to C_0(L,M)$ be a $^*$-derivation. Given $\varepsilon>0$, and a compact subset $K\subset L$, then there exists a continuous (and bounded) function $Z_K : K \to M$ such that  $\| Z_K \| \leq \frac{1+2 \varepsilon}{2} \|D\|$, $Z_K^*=-Z_K$, and $$D(X) (t) = [Z_K,X] (t),$$ for every $X\in C_0(L,M)$ and every $t\in K$. If $D$ is a general derivation on $C_0(L,M)$, then for each $\varepsilon>0$, and each compact subset $K\subset L$, there exists a continuous (and bounded) function $Z_K : K \to M$ such that $\| Z_K \| \leq ({1+2 \varepsilon}) \|D\|$ and $D(X) (t)= [Z_K , X](t)$, for every $X\in C_0(L,M)$ and every $t\in K$.
\end{theorem}

\begin{proof} We shall only prove the first statement, the second one is a straight consequence of it. For this purpose, let $D$ be a $^*$-derivation on $C_0(L,M)$. When $L$ is compact the conclusion follows from Theorem \ref{t AkJohn strengtened with control on the norm}. We can thus assume that $L$ is non-compact. Let $K$ be a compact subset of $L$. Let us fix an open susbset $O$ such that $K\subset O\subsetneq L$. We also fix an arbitrary function $f_{K,O}\in C_0(L)$ with $0\leq f_{K,O}\leq 1$, $f_{K,O}|_{K}\equiv 1$ and $f_{K,O}|_{L\backslash O} \equiv 0$.  Following the arguments in the proof of Theorem \ref{t AkJohn strengtened with control on the norm}, and by applying Lemma \ref{l weak-2-local derivation anihilates on deltat} and \eqref{eq weak2der with Gammafko} in the proof of the latter, we deduce that, for each $t\in K$, the mapping $\delta_t D \Gamma_{f_{K,O}} : M\to M$ defines a $^*$-derivation on $M$, and the mapping $\Upsilon_{f_{K,O}} : {K} \to \hbox{$^*$-Der} (M),$ $t\mapsto \Upsilon_{f_{K,O}} (t)=\delta_t D \Gamma_{f_{K,O}}$ is a $\tau$-to-(point-norm) continuous mapping, where $\tau$ stands for the topology of ${L}$ and for its restriction to $K$.
Consequently, $\Upsilon_{f_{K,O}} (K)$ is a point-norm compact subset in Der$(M)$. Theorem 2.1 in \cite{AkJohn79} assures that $\Upsilon_{f_{K,O}} (K)$ is norm compact, and hence the point-norm and the norm topologies coincide on $\Upsilon_{f_{K,O}} (\Omega)$.\smallskip

Sakai's theorem (cf. \cite[Theorem 4.1.6]{S}) assures that every derivation on $M$ is inner and the mapping $\theta: i M_{sa}/Z(M_{sa}) \to \hbox{$^*$-Der} (M)$, $a+Z(M_{sa})\mapsto \hbox{ad}_a$ is an isomorphism of Banach spaces. Thus, the mapping $\theta^{-1} \Upsilon_{f_{K,O}} : K \to i M_{sa}/Z(M_{sa})$ is continuous.\smallskip

Let us define a mapping $\Upsilon : {L} \to i M_{sa}/Z(M_{sa})$ given by the following rules: For each $t$ in $L$ let $K\subset O\subsetneq L$ with $O$ open and $K$ compact and $t\in K$, let $f_{K,O}$ be a function in $C_0(L)$ with $0\leq f_{K,O}\leq 1$, $f_{K,O}|_{K}\equiv 1$ and $f_{K,O}|_{L\backslash O} \equiv 0$. We set $\Upsilon (t) :=\theta^{-1} \Upsilon_{f_{K,O}} (t)$. We claim that $\Upsilon$ is well-defined. Indeed, suppose $K_i\subset O_i\subsetneq L$ with $O_i$ open and $K_i$ compact, $t\in K_i$, $f_{K_i,O_i}\in C_0(L)$ with $0\leq f_{K_i,O_i}\leq 1$, $f_{K_i,O_i}|_{K_i}\equiv 1$ and $f_{K_i,O_i}|_{L\backslash O_i} \equiv 0$, for $i=1,2$. By Lemma \ref{l weak-2-local derivation anihilates on deltat} we have $$ \Upsilon_{f_{K_i,O_i}} (s) (X(s)) = \delta_s D \Gamma_{f_{K_i,O_i}} (X(s)) = \delta_s D (X),$$ for every $s\in K_i$, every $i=1,2$, and every $X\in C_0(L,M)$. Therefore, $\Upsilon_{f_{K_1,O_1}} (s) = \Upsilon_{f_{K_2,O_2}} (s)$ for every $s\in K_1\cap K_2$, which proves the claim.\smallskip

Unfortunately, a locally compact Hausdorff space need not be paracompact, so Michael's selection principle cannot be applied to our mapping $\Upsilon$ as we did in the proof of Theorem \ref{t AkJohn strengtened with control on the norm}. However, for each compact subset $K\subset L$, applying the same arguments we gave in the final paragraph of the proof of Theorem \ref{t AkJohn strengtened with control on the norm} to $\Upsilon|_{K}$, we can find, via Michael's selection principle (see \cite[Theorem 3.2]{Michael56}) and Lemma \ref{l weak-2-local derivation anihilates on deltat}, a continuous function $Z_{K}: K \to i M_{sa}$ satisfying $Z_{K} \in \theta^{-1} \Upsilon (t)$ for every $t\in K$, $\|Z_{K} \| \leq \frac{1+2\varepsilon}{2} \|D\|$, and  \begin{equation}\label{eq representation by ZfKO} D(X) (t) = [Z_{K},X] (t),
\end{equation} for every $X\in C_0(L,M)$ and every $t\in K$.
\end{proof}

The previous Theorem \ref{t representation of derivations on C0LM} only produces local representations for derivations on $C_0(L,M)$, where $M$ is a von Neumann algebra and $L$ is a locally compact Hausdorff space. If we replace $M$ with $B(H)$ (or with $K(H)$), we can obtain a more global representation at the cost of loosing certain continuity on the mapping $Z_0: L\to B(H)$ that represents our derivation. When dealing with $K(H)$ we also find the obstacle that, Akeman-Johnson's theorem, asserting that on Der$(M)$ point-norm compactness and norm compactness are equivalent notions \cite[Theorem 2.1]{AkJohn79}, is only valid for von Neumann algebras.\smallskip

E.C. Lance considers in \cite[\S 2]{Lanc69} the following related problem. Let $\Omega$ be a separable compact Hausdorff space and let $H$ be a (separable) infinite dimensional Hilbert space. Then every derivation on $C(\Omega)\otimes B(H)$ is inner (see \cite[Theorem 2.4]{Lanc69}). In \cite[Lemma 2.1]{Lanc69} it is implicitly proved a result which was later materialized in \cite[Theorem 3.4]{AkEllPedTomi76} in the following terms: Let $\Gamma$ be a {separated} locally compact Hausdorff space, then every derivation on $C_0(\Gamma,K(H))$ is inner in $M(C_0(\Gamma,K(H)))$, that is, every derivation on $C_0(\Gamma,K(H))$ is determined by a multiplier. {In Theorem \ref{t representation of derivations on C0LKH} and Corollary \ref{c tau-strong* continuity of Z0 in K(H)} we show that this multiplier can be chosen bounded by a multiple of the norm of the derivation}. In Theorems \ref{t representation of derivations on C0LBH} and \ref{thm final tau-strong* continuity of Z0} a similar conclusion is proved for derivations on $C_0(L,B(H))$. Thus, our main conclusions connect the results in this paper with those previously obtained by Akemann, Elliott, Lance, Pedersen and Tomiyama, which have been commented above.\smallskip

Following the usual notation, the set of all finite dimensional subspaces of a complex Hilbert space $H$ will be denoted by $ \mathfrak{F} (H).$ We consider in $\mathfrak{F} (H)$ the natural order given by inclusion. For each $F\in \mathfrak{F} (H)$, $p_{_F}$ will denote the orthogonal projection of $H$ onto $F$.

\begin{theorem}\label{t representation of derivations on C0LKH} Let $H$ be a complex Hilbert space, let $L$ be a locally compact Hausdorff space whose topology is denoted by $\tau$, and let $D$ be a $^*$-derivation on $ C_0(L,K(H))$. Then there exists a (bounded) mapping $Z_0: L \to B(H)$ which is $\tau$-to-$\sigma(B(H),B(H)_*)$ continuous, $\| Z_0 \| \leq \|D\|$, $Z_0^*(t)=-Z_0 (t)$, for every $t\in L$, and $$D(X)= [Z_0,X],$$ for every $X\in C_0(L,K(H))$. If $D$ is a general derivation on $C_0(L,K(H))$, then  there exists a bounded function $Z_0 : L \to B(H)$ which is $\tau$-to-$\sigma(B(H),B(H)_*)$ continuous, $\| Z_0 \| \leq 2 \|D\|$ and $D(X) = [Z_0 , X] $, for every $X\in C_0(L,K(H))$.
\end{theorem}

\begin{proof} We shall only prove the first statement. For each $F\in \mathfrak{F} (H)$, we denote by $\widehat{p}_{_F}=1\otimes p_{_F}$ the constant function mapping each $t$ in $L$ to $p_{_F}$. To simplify the notation we write $C$ for the C$^*$-algebra $C_0(L,K(H))$. The mapping $\widehat{p}_{_F} D \widehat{p}_{_F}|_{\widehat{p}_{_F} C \widehat{p}_{_F}} : \widehat{p}_{_F} C \widehat{p}_{_F} \to \widehat{p}_{_F} C \widehat{p}_{_F}$, $x\to \widehat{p}_{_F}D (\widehat{p}_{_F} x \widehat{p}_{_F}) \widehat{p}_{_F}$ is a $^*$-derivation on $\widehat{p}_{_F} C \widehat{p}_{_F}$ (compare \cite[Proposition 2.7]{NiPe2014}). Since $\widehat{p}_{_F} C \widehat{p}_{_F}\cong C_0(L,M_n)$, by Lemma \ref{l C0(LM_n) has the idp control of the norm}, there exists $Z_{_F}\in C_b(L,p_{_F}B(H) p_{_F})$ with $\|Z_{_F}\| \leq \|D\|$, $Z_{_F}^* =- Z_{_F}$ and \begin{equation}\label{eq identity for pF} \widehat{p}_{_F} D (\widehat{p}_{_F} X\widehat{p}_{_F} ) \widehat{p}_{_F}  = [Z_F,\widehat{p}_{_F} X \widehat{p}_{_F}],
 \end{equation} for every $X\in C_0(L,B(H))$.\smallskip

For each $t\in L$, the net $(Z_{_F}(t))_{_{F\in\mathfrak{F}(H)}}\subset B(H)$ is bounded, and hence we can find a subnet $(Z_{_F}(t))_{_{F\in\Lambda}}$ converging to some $Z_0(t)=-Z_0^*(t)\in B(H)$ in the weak$^*$ topology of $B(H)$ with $\|Z_0 (t)\|\leq \|D\|$. We observe that the chosen subnet depends on the point $t$. In any case, the net $(p_{_F})_{_{F\in\Lambda}}$ converges to the unit of $B(H)$ in the strong$^*$ topology of the latter von Neumann algebra. We define this way a bounded map $Z_0 : L \to  B(H)$. We fix a point $t_0\in L$. By \eqref{eq identity for pF} and Lemma \ref{l weak-2-local derivation anihilates on deltat} we have \begin{equation}\label{eq identity June9}
 {p}_{_F} D (f_{\{t_0\},O}\otimes ({p}_{_F} X(t_0) {p}_{_F}) ) (t_0)  {p}_{_F}= {p}_{_F} D (\widehat{p}_{_F} X\widehat{p}_{_F} )(t_0)  {p}_{_F}  = [Z_F (t_0), {p}_{_F} X(t_0) {p}_{_F}],
 \end{equation} for every $X\in C_0(L,K(H))$, where $f_{\{t_0\},O}$ satisfies the obvious conditions.\smallskip

Now we fix an arbitrary $X\in C_0(L,K(H))$. Let $(Z_{_F}(t_0))_{_{F\in\Lambda}}$ be the corresponding subnet converging to $Z_0(t_0)$ in the weak$^*$ topology of $B(H)$. The net $(p_{_F})_{_{F\in\Lambda}}$ is an approximate unit in $K(H)$. Therefore, the net $({p}_{_F} X(t_0) {p}_{_F})_{_{F\in\Lambda}}$ converges in the norm topology of $K(H)$ to $X(t_0)$. It is not hard to see that $(f_{\{t_0\},O}\otimes ({p}_{_F} X(t_0) {p}_{_F}) )_{_{F\in\Lambda}}$ converges in norm to $f_{\{t_0\},O}\otimes X(t_0)$, and the continuity of $D$ implies that $(D (f_{\{t_0\},O}\otimes ({p}_{_F} X(t_0) {p}_{_F}) ))_{_{F\in\Lambda}}\to D(f_{\{t_0\},O}\otimes X(t_0))$ in norm. Similar arguments show that $$({p}_{_F} D (f_{\{t_0\},O}\otimes ({p}_{_F} X(t_0) {p}_{_F}) ) (t_0)  {p}_{_F})_{_{F\in\Lambda}} \to D(f_{\{t_0\},O}\otimes X(t_0)) (t_0) =D(X) (t_0)$$ in norm, where in the last equality we apply Lemma \ref{l weak-2-local derivation anihilates on deltat}. That is an appropriate subnet of the left-hand side of \eqref{eq identity June9} tends to $D(X) (t_0)$ in norm.\smallskip

Concerning the right-hand side of \eqref{eq identity June9}, we observe that $(Z_{_F}(t_0))_{_{F\in\Lambda}}\to Z_0(t_0)$ in the weak$^*$ topology of $B(H)$, and as before $({p}_{_F} X(t_0) {p}_{_F})_{_{F\in\Lambda}}\to X(t_0)$ in norm. It is known that in these circumstances, $([Z_F (t_0), {p}_{_F} X(t_0) {p}_{_F}])_{_{F\in\Lambda}} \to [Z_0(t_0), X(t_0)]$ in the weak$^*$ topology of $B(H)$ (compare \cite[Lemma 2.7]{JorPe}). Thus, we deduce from \eqref{eq identity June9} that the identity \begin{equation}\label{eq puntual identification Jun 9}  D(X) (t_0) = [Z_0(t_0), X(t_0)],
\end{equation} holds for every $X\in C_0(L,K(H))$ and every $t_0\in L$.\smallskip

The $\tau$-to-$\sigma(B(H),B(H)_*)$ continuity of $Z_0$ can be deduced as in the final paragraph of \cite[Proposition 2.10]{JorPe}.
\end{proof}

Let us note that Remark 2.11 in \cite{JorPe} shows that the mapping $Z_0:L\to B(H)$ given by Theorem \ref{t representation of derivations on C0LKH} need not be, in general, $\tau$-to-norm continuous.\smallskip

We have already commented that Sakai's theorem assures that every derivation on a C$^*$-algebra is continuous \cite{Sak60}. J.R. Ringrose proved in \cite{Ringrose72} that actually every derivation from a C$^*$-algebra $A$ into a Banach $A$-bimodule is continuous. Let us revisit other additional properties of derivations. Another result due to S. Sakai implies that every derivation on a von Neumann algebra $M$ is inner \cite[Theorem 4.1.6]{S}. It is also due to the same author that the product of a von Neumann algebra is separately weak$^*$-continuous (see \cite[Theorem 1.7.8]{S}). It follows from the last results that every derivation on a von Neumann algebra is weak$^*$-continuos.\smallskip

In general, given a von Neumann algebra $M$ and a locally compact Hausdorff space $L$, the C$^*$-algebra is not a dual Banach space, however we have some other topologies which are weaker than the norm topology. We shall consider here the``\emph{point-norm}'' and the ``\emph{point-weak$^*$}'' topologies on $C_0(L,M)$. We recall that a net $(X_{\lambda})$ in $C_0(L,M)$ converges to an element $X\in C_0(L,M)$ in the point-weak$^*$ topology (respectively, in the point-norm topology) if for each $t\in L$, the net $(X_{\lambda}(t))$ converges to $X(t)$ in the weak$^*$ topology (respectively, in the norm topology) of $M$. We prevent the reader that the term ``point-norm'' has been used with another meaning in the setting of operators, however we consider that the dual use does not produce any contradiction in this note.\smallskip

The local representation given in Theorem \ref{t representation of derivations on C0LM} will be applied to prove the following:

\begin{proposition}\label{p automatic point-weak* continuity} Let $M$ be a von Neumann algebra, and let $L$ be a locally compact Hausdorff space. Then every derivation on $C_0(L,M)$ is point-weak$^*$ continuous.
\end{proposition}

\begin{proof} Let $(X_\lambda)_{\lambda}$ be a net in $C_0(L,M)$ converging to some $X\in C_0(L,M)$ in the point-weak$^*$ topology. Let $t_0$ be a point in $L$. Let us fix a compact set $K\subset L$ with $t_0\in K$. By Theorem \ref{t representation of derivations on C0LM} there exists a continuous and bounded function $Z_K : K \to M$ such that  $\| Z_K \| \leq 2 \|D\|$, and \begin{equation}
\label{eq represent Zk 10 06} D(X) (t) = [Z_K,X] (t),
\end{equation} for every $X\in C_0(L,M)$ and every $t\in K$. It follows from the assumptions on $(X_\lambda)_{\lambda}$ that $(X_\lambda (t_0)) _{\lambda}\to X (t_0)$ in the weak$^*$ topology of $M$. We know from \eqref{eq represent Zk 10 06} that $ (D (X_\lambda) (t_0)) )_{\lambda} = \left( [Z_K (t_0), X_\lambda (t_0)] \right), $ for all $\lambda$. The separate weak$^*$ continuity of the product of $M$ implies that $$ (D (X_\lambda) (t_0) ) )_{\lambda} = \left( [Z_K (t_0), X_\lambda (t_0)] \right)_{\lambda} \to  [Z_K (t_0), X (t_0)] = D (X) (t_0)$$ in the weak$^*$ continuity of $M$, which concludes the proof.
\end{proof}

We are now in position to establish a global representation theorem for derivations on $C_0(L,B(H))$.

\begin{theorem}\label{t representation of derivations on C0LBH} Let $H$ be a complex Hilbert space, and let $L$ be a locally compact Hausdorff space whose topology is denoted by $\tau$. Suppose $D$ is a $^*$-derivation on $ C_0(L,B(H))$. Then there exists a (bounded) mapping $Z_0: L \to B(H)$ which is $\tau$-to-$\sigma(B(H),B(H)_*)$ continuous, $\| Z_0 \| \leq \|D\|$, $Z_0^*(t)=-Z_0 (t)$, for every $t\in L$, and $$D(X)= [Z_0,X],$$ for every $X\in C_0(L,B(H))$. If $D$ is a general derivation on $C_0(L,B(H))$, then  there exists a bounded function $Z_0 : L \to B(H)$ which is $\tau$-to-$\sigma(B(H),B(H)_*)$ continuous, $\| Z_0 \| \leq 2 \|D\|$ and $D(X) = [Z_0 , X] $, for every $X\in C_0(L,B(H))$.
\end{theorem}

\begin{proof} Let $D$ be a $^*$-derivation on $C_0(L,B(H))$. The subalgebra $C_0(L,K(H))$ is a norm closed (two-sided) ideal of $C_0(L,B(H))$. Lemma 3.4 in \cite{NiPe2015} guarantees that $D(C_0(L,K(H))) \subset C_0(L,K(H))$ and $D|_{C_0(L,K(H))} : C_0(L,K(H))\to C_0(L,K(H))$ is a $^*$-derivation (This also can be deduced from Theorem \ref{t representation of derivations on C0LM} above). By Theorem \ref{t representation of derivations on C0LKH} there exists a bounded mapping $Z_0: L \to B(H)$ which is $\tau$-to-$\sigma(B(H),B(H)_*)$ continuous, $\| Z_0 \| \leq \|D\|$, $Z_0^*(t)=-Z_0 (t)$, for every $t\in L$, and \begin{equation}\label{eq represetation on compact Jun 11} D(X)= [Z_0,X],
\end{equation} for every $X\in C_0(L,K(H))$.\smallskip

We shall show that the identity $D(X)= [Z_0,X]$ also holds for every $X\in C_0(L,B(H))$. It is enough to prove that $D(X) (t)= [Z_0(t),X (t)]$  for all  $X\in C_0(L,B(H))$ and all $t\in L$. Fix $Y\in C_0(L,B(H))$ and $t_0$ in $L$. We also pick a compact subset $K$ with $t_0$ in $K$, an open subset $O\subsetneq L$ containing $K$, and a continuous function $f_{K,O}\in C_0(L)$ with $0\leq f_{K,O}\leq 1$, $f_{K,O}|_{K}\equiv 1$ and $f_{K,O}|_{L\backslash O} \equiv 0$. By Lemma \ref{l weak-2-local derivation anihilates on deltat} we know that $D(Y) (t_0) = D( f_{K,O}\otimes Y(t_0))(t_0)$.\smallskip

The element $Y(t_0)\in B(H)$ can be approximated in the weak$^*$ topology of $B(H)$ by a net $(k_{\lambda})_{\lambda}\subset K(H).$ Clearly, the net $(f_{K,O}\otimes k_{\lambda})_{\lambda}$ lies in $C_0(L,K(H))$ and converges to $f_{K,O}\otimes Y(t_0)$ in the point-weak$^*$ topology. Proposition \ref{p automatic point-weak* continuity} implies that $$(D(f_{K,O}\otimes k_{\lambda}))_{\lambda} \to D(f_{K,O}\otimes Y(t_0))$$ in the point-weak$^*$ topology. Applying \eqref{eq represetation on compact Jun 11} we get $$(D(f_{K,O}\otimes k_{\lambda}))_{\lambda} = ([ Z_0 , f_{K,O}\otimes k_{\lambda}])_{\lambda}\to [Z_0, f_{K,O}\otimes Y(t_0)],$$ in the point-weak$^*$ topology, which proves that $$D(f_{K,O}\otimes Y(t_0)) = [Z_0, f_{K,O}\otimes Y(t_0)],$$ and thus
$$[Z_0,Y](t_0) = [Z_0, f_{K,O}\otimes Y(t_0)](t_0)= D(f_{K,O}\otimes Y(t_0)) (t_0)  =D(Y) (t_0).$$
\end{proof}

We shall finally show that in Theorem \ref{t representation of derivations on C0LBH} the conclusion concerning the continuity of the mapping $Z_0$ can be improved.

\begin{theorem}\label{thm final tau-strong* continuity of Z0} Let $H$ be a complex Hilbert space, let $L$ be a locally compact Hausdorff space whose topology is denoted by $\tau$, and let $D$ be a derivation on $C_0(L,B(H))$. Then there exists a bounded function $Z_0 : L \to B(H)$ which is $\tau$-to-norm continuous {\rm(}that is $Z_0\in M (C_0(L,B(H)))${\rm)}, $\| Z_0 \| \leq 2 \|D\|$ and $D(X) = [Z_0 , X] $, for every $X\in C_0(L,B(H))$.
\end{theorem}

\begin{proof} Everything except the $\tau$-to-norm continuity of $Z_0$ has been proved in Theorem \ref{t representation of derivations on C0LBH}. We shall show that $Z_0$ is $\tau$-to-norm continuous. Let us fix $t_0\in L$ and a compact neighborhood $K$ of $t_0$. By Theorem \ref{t representation of derivations on C0LM} there exists $Z_K\in C(K,B(H))$ such that $[Z_K,X](t)=[Z_0,X](t)$ for all $t\in K$, $X\in C_0(L,B(H))$. This implies that $Z_0 (t) - Z_K (t) = \alpha(t) I$ for all $t\in K$, where $\alpha$ is a mapping from  $K$ to $\mathbb{C}$ and $I$ is the identity operator on $H$ (this can be easily checked by evaluating constant functions on $K$). The $\tau$-to-weak$^*$ continuity of $Z_0$ combined with the $\tau$-to-norm continuity of $Z_K$ show that $\alpha:K\to \C$ is $\tau$-to-norm continuous at $t_0$. $Z_0|_{K} =Z_K+\alpha(.) I$ is $\tau$-to-norm continuous at $t_0$. Since $K$ is a compact neighborhood of $t_0$, we conclude that $Z_0$ is continuous at $t_0$.\end{proof}

\begin{corollary}\label{c tau-strong* continuity of Z0 in K(H)} Under the assumptions in Theorem \ref{t representation of derivations on C0LKH}, for each derivation $D$ on $C_0(L,K(H))$, there exists a bounded function $Z_0 : L \to B(H)$ which is $\tau$-to-$s^*(B(H),B(H)_*)$ continuous {\rm(}i.e. $Z_0\in M(C_0(L,K(H)))${\rm)}, $\| Z_0 \| \leq 2 \|D\|$ and $D(X) = [Z_0 , X] $, for every $X\in C_0(L,K(H))$.
\end{corollary}

\begin{proof}
Let $D$ be a derivation on $C_0(L,K(H))$ and let $Z_0$ be the mapping given by Theorem \ref{t representation of derivations on C0LKH}. By Sakai's theorem (see  \cite{Sak71a}) $D$ extends to a derivation $\widetilde{D}$ on $M(C_0(L,K(H))) =C_b(L,(B(H),s^*))$.  Corollary 3.5 in \cite{AkEllPedTomi76} affirms that $\widetilde{D}$ is inner. That is, there exists a bounded function $Z_2 : L \to B(H)$ which is $\tau$-to-$s^*(B(H),B(H)_*)$ continuous {\rm(}i.e. $Z_2\in M(C_0(L,K(H)))${\rm)}, $\widetilde{D}(X) = [Z_2 , X] $, for every $X\in M(C_0(L,K(H)))$.\smallskip

Therefore, $[Z_0-Z_2, X] =0,$ for every $X\in C_0(L,K(H))$, and hence $(Z_0-Z_2)(t) = \alpha (t) I$, for a suitable mapping $\alpha: L\to \mathbb{C}$. We deduce from the $\tau$-to-weak$^*$ continuity of $Z_0$ and the $\tau$-to-$s^*(B(H),B(H)_*)$ continuity of $Z_2$ that $\alpha$ is $\tau$-to-norm continuous. Finally, the identity $Z_0 = Z_2 +\alpha (.) I$ proves that $Z_0$ is  $\tau$-to-$s^*(B(H),B(H)_*)$ continuous.
\end{proof}

\noindent\textbf{Acknowledgements} We would like to thank our colleague Prof. J. Bonet for his useful suggestions during the preparation of this note. A part of this work was done during the visit of the second author to Universitat Polit\'ecnica de Valencia in Alcoy and to the IUMPA in Valencia. He would like to thank the Department of Mathematics of the Universitat Polit\'ecnica de Valencia and the first author for the hospitality.

\end{document}